\documentclass[12pt,a4paper]{amsart}
\usepackage{a4wide}

\newtheorem{theorem}{Theorem}[section]
\newtheorem{lemma}[theorem]{Lemma}
\newtheorem{proposition}[theorem]{Proposition}

\numberwithin{equation}{section}

\begin{document}

\title[Cartan geometry and nef tangent bundle]{Holomorphic Cartan 
geometry on manifolds\\[5pt] with numerically effective tangent bundle}

\author[I. Biswas]{Indranil Biswas}

\address{School of Mathematics, Tata Institute of Fundamental
Research, Homi Bhabha Road, Bombay 400005, India}

\email{indranil@math.tifr.res.in}

\author[U. Bruzzo]{Ugo Bruzzo}

\address{Scuola Internazionale Superiore di Studi Avanzati,
Via Beirut 2--4, 34013, Trieste, and Istituto Nazionale di
Fisica Nucleare, Sezione di Trieste, Italy}

\email{bruzzo@sissa.it}



\date{}

\begin{abstract}
Let $X$ be a compact connected K\"ahler manifold such
that the holomorphic tangent bundle $TX$ is numerically effective.
A theorem of \cite{DPS} says that there is a finite unramified Galois 
covering $M\, \longrightarrow\, X$, a
complex torus $T$, and a holomorphic surjective submersion
$f\, :\, M\,\longrightarrow\,T$, such that the fibers of $f$ are
Fano manifolds with numerically effective tangent bundle.
A conjecture of Campana and Peternell says that the fibers of
$f$ are rational and homogeneous.
Assume that $X$ admits a holomorphic Cartan geometry. We prove that
the fibers of $f$ are rational homogeneous varieties.
We also prove that the holomorphic principal
${\mathcal G}$--bundle over $T$ given by $f$, where $\mathcal G$
is the group of all holomorphic automorphisms of a fiber, admits
a flat holomorphic connection.
\end{abstract}

\maketitle

\section{Introduction}\label{intro.}

Let $X$ be a compact connected K\"ahler manifold such that the
holomorphic tangent bundle $TX$ is numerically effective. (The
notions of numerically effective vector bundle and numerically
flat vector bundle over a compact K\"ahler manifold were
introduced in \cite{DPS}.) From a theorem of Demailly, Peternell
and Schneider we know that there is a finite unramified
Galois covering
$$
\gamma\, :\, M\, \longrightarrow\, X\, ,
$$
a complex torus $T$, and a holomorphic surjective submersion
$$
f\, :\, M\,\longrightarrow\,T\, ,
$$
such that the fibers of $f$ are
Fano manifolds with numerically effective tangent bundle (see
\cite[p. 296, Main~Theorem]{DPS}). It is conjectured
by Campana and Peternell that the fibers of $f$ are
rational homogeneous varieties (i.e., varieties
of the form ${\mathcal G}/P$, where $P$ is
a parabolic subgroup of a complex semisimple group ${\mathcal G}$)
\cite[p. 170]{CP}, \cite[p. 296]{DPS}. Our aim here is to verify this 
conjecture under the extra assumption that $X$ admits a holomorphic 
Cartan geometry.

Let $(E'_H\, ,\theta')$
be a holomorphic Cartan geometry on $X$ of type $G/H$, where
$H$ is a complex Lie subgroup of a complex Lie group $G$.
(The definition of Cartan geometry is recalled in
Section \ref{sec2}.) Consider the pullback $\theta$ of
$\theta'$ to the holomorphic principal $H$--bundle
$E_H\,:=\, \gamma^*E'_H$, where $\gamma$ is the above
covering map. The pair $(E_H\, ,\theta)$ 
is a holomorphic Cartan geometry on $M$. Using
$(E_H\, ,\theta)$ we prove the following theorem (see
Theorem \ref{thm1}):

\begin{theorem}\label{thm0}
There is a semisimple linear algebraic group $\mathcal 
G$ over $\mathbb C$, a parabolic subgroup $P\, \subset\,
\mathcal G$, and a holomorphic principal $\mathcal G$--bundle
$$
{\mathcal E}_{\mathcal G}\, \longrightarrow\, T\, ,
$$
such that the fiber bundle ${\mathcal E}_{\mathcal G}/P\,
\longrightarrow\, T$ is holomorphically isomorphic to the
fiber bundle $f\, :\, M\, \longrightarrow\, T$.
\end{theorem}

The group $\mathcal G$ in Theorem \ref{thm0} is the group
of all holomorphic automorphisms of a fiber of $f$.
Let ${\rm ad}({\mathcal E}_{\mathcal G})
\,\longrightarrow\, T$ be the adjoint vector bundle of the
principal $\mathcal G$--bundle ${\mathcal E}_{\mathcal G}$
in Theorem \ref{thm0}.
Let $K^{-1}_f\,\longrightarrow\, M$ be the relative
anti--canonical line bundle for the projection $f$.

We prove the following (see Proposition \ref{lem1} and Proposition
\ref{lem3}):

\begin{proposition}\label{lem0}
Let $X$ is a compact connected K\"ahler manifold such
that $TX$ is numerically effective, and let
$(E'_H\, ,\theta')$
be a holomorphic Cartan geometry on $X$ of type $G/H$. Then
the following two statements hold:
\begin{enumerate}
\item The adjoint vector bundle ${\rm ad}({\mathcal E}_{\mathcal G})$
is numerically flat.

\item The principal $\mathcal G$--bundle
${\mathcal E}_{\mathcal G}$ admits a flat holomorphic connection.
\end{enumerate}
\end{proposition}

\section{Cartan Geometry and numerically effectiveness}\label{sec2}

Let $G$ be a connected complex Lie group. Let
$
H\, \subset\, G
$
be a connected complex Lie subgroup. The Lie algebra of
$G$ (respectively, $H$) will be denoted by $\mathfrak g$
(respectively, $\mathfrak h$).

Let $Y$ be a connected
complex manifold. The holomorphic tangent bundle of $Y$ will be
denoted by $TY$.
Let $E_H\, \longrightarrow\, Y$ be a holomorphic principal
$H$--bundle. For any $g\,\in\, H$, let
\begin{equation}\label{betah}
\beta_g\,:\, E_H\, \longrightarrow\, E_H
\end{equation}
be the biholomorphism defined by $z\, \longmapsto\, zg$.
For any $v\, \in\, {\mathfrak h}$, let
\begin{equation}\label{z}
\zeta_v\, \in\, H^0(E_H,\, TE_H)
\end{equation}
be the holomorphic vector field on $E_H$ associated
to the one--parameter family of biholomorphisms
$t\,\longmapsto \, \beta_{\exp(tv)}$.
Let
$$
\text{ad}(E_H)\, :=\, E_H\times^H{\mathfrak h}\,\longrightarrow\,Y
$$
be the adjoint vector bundle associated $E_H$ for the
adjoint action of $H$ on $\mathfrak h$. The adjoint vector
bundle of a principal $G$--bundle is defined similarly.

A \textit{holomorphic Cartan geometry} of type $G/H$
on $Y$ is a holomorphic principal $H$--bundle
\begin{equation}\label{e0}
p\, :\, E_H\, \longrightarrow\, Y
\end{equation}
together with a $\mathfrak g$--valued holomorphic one--form
\begin{equation}\label{e00}
\theta\, \in\, H^0(E_H,\, \Omega^1_{E_H}\otimes_{\mathbb C}
{\mathfrak g})
\end{equation}
satisfying the following three conditions:
\begin{enumerate}
\item $\beta^*_g\theta\, =\, \text{Ad}(g^{-1})\circ\theta$
for all $g\, \in\,H$, where $\beta_g$ is defined in \eqref{betah},

\item $\theta(z)(\zeta_v(z)) \, =\, v$ for all
$v\, \in\, {\mathfrak h}$ and $z\, \in\, E_H$ (see \eqref{z}
for $\zeta_v$), and

\item for each point $z\, \in\, E_H$, the homomorphism from the
holomorphic tangent space
\begin{equation}\label{e-1}
\theta(z) \,:\, T_zE_H\, \longrightarrow\, {\mathfrak g}
\end{equation}
is an isomorphism of vector spaces.
\end{enumerate}
(See \cite{S}.)

A holomorphic line bundle $L\,\longrightarrow\, Y$ is called
\textit{numerically effective} if $L$ admits Hermitian structures such
that the negative part of the curvatures are arbitrarily small
\cite[p. 299, Definition 1.2]{DPS}. If $Y$ is a projective manifold, then 
$L$ is numerically effective if and only if the restriction of it
to every complete curve has nonnegative degree. A holomorphic vector 
bundle
$E\,\longrightarrow \, Y$ is called \textit{numerically effective} if
the tautological line bundle ${\mathcal O}_{{\mathbb P}(E)}(1)
\,\longrightarrow\, {\mathbb P}(E)$ is numerically effective.

Let $X$ be a compact connected K\"ahler manifold such that the 
holomorphic tangent bundle $TX$ is numerically effective. Then there is 
a finite \'etale Galois covering
\begin{equation}\label{e1}
\gamma\, :\, M\, \longrightarrow\, X\, ,
\end{equation}
a complex torus $T$ and a holomorphic surjective submersion
\begin{equation}\label{e2}
f\, :\, M\, \longrightarrow\, T
\end{equation}
such that the fibers of $f$ are connected Fano manifolds with
numerically effective tangent bundle \cite[p. 296, Main~Theorem]{DPS}.

\begin{theorem}\label{thm1}
Let $(E'_H\, ,\theta')$
be a holomorphic Cartan geometry on $X$ of type $G/H$,
where $X$ is a compact connected K\"ahler manifold such that the
holomorphic tangent bundle $TX$ is numerically effective. Then
there is
\begin{enumerate}
\item a semisimple linear algebraic group $\mathcal 
G$ over $\mathbb C$,

\item a parabolic subgroup $P\, \subset\, \mathcal G$, and

\item a holomorphic principal $\mathcal G$--bundle
${\mathcal E}_{\mathcal G}\, \longrightarrow\, T$,
\end{enumerate}
such that the fiber bundle ${\mathcal E}_{\mathcal G}/P\,
\longrightarrow\, T$ is holomorphically isomorphic to the
fiber bundle $f$ in \eqref{e2}.
\end{theorem}

\begin{proof}
Let 
\begin{equation}\label{f1}
(E_H\, ,\theta)
\end{equation}
be the holomorphic Cartan geometry on $M$ obtained by pulling back
the holomorphic Cartan geometry $(E'_H\, ,\theta')$ on $X$ using the
projection $\gamma$ in \eqref{e1}.

Let
\begin{equation}\label{eg}
E_G\, :=\, E_H\times^H G \, \longrightarrow\, M
\end{equation}
be the holomorphic principal $G$--bundle obtained by extending the
structure group of $E_H$ using the inclusion of $H$ in $G$. So
$E_G$ is a quotient of $E_H\times G$, and two points $(z_1\, ,g_1)$
and $(z_2\, ,g_2)$ of $E_H\times G$ are identified in $E_G$ if
there is an element $h\, \in\, H$ such that $z_2\,=\, z_1h$
and $g_2\,=\, h^{-1}g_1$. Let
$$
\theta_{\text{MC}}\, :\,
TG\, \longrightarrow\, G\times\mathfrak g
$$
be the $\mathfrak g$--valued Maurer--Cartan
one--form on $G$ constructed using
the left invariant vector fields.
Consider the $\mathfrak g$--valued holomorphic
one--form
$$
\widetilde{\theta}\, :=\, p^*_1 \theta + p^*_2\theta_{\text{MC}}
$$
on $E_H\times G$, where $p_1$ (respectively,
$p_2$) is the projection of $E_H\times G$ to
$E_H$ (respectively, $G$), and $\theta$ is
the one--form in \eqref{f1}. This form $\widetilde{\theta}$
descends to a $\mathfrak g$--valued holomorphic one--form on the
quotient space $E_G$ in \eqref{eg}, and the descended form
defines a holomorphic connection on $E_G$; see \cite{At2}
for holomorphic connection.
Therefore, the principal $G$--bundle $E_G$ in \eqref{eg} is
equipped with a holomorphic connection.
This holomorphic connection on $E_G$ will be denoted by $\nabla^G$.

The inclusion map ${\mathfrak h}\,\,\hookrightarrow\,\mathfrak g$
produces an inclusion
$$
\text{ad}(E_H)\,\hookrightarrow\, \text{ad}(E_G)
$$
of holomorphic vector bundles.
Using the form $\theta$, the quotient 
bundle $\text{ad}(E_G)/\text{ad}(E_H)$
gets identified with the holomorphic tangent bundle $TM$. Therefore,
we get a short exact sequence of holomorphic vector bundles on $M$
\begin{equation}\label{eq1}
0\,\longrightarrow\, \text{ad}(E_H)\,\longrightarrow\,\text{ad}(E_G)
\,\longrightarrow\, TM \,\longrightarrow\, 0\, .
\end{equation}

The holomorphic connection $\nabla^G$ on $E_G$ induces a holomorphic 
connection on the adjoint vector bundle $\text{ad}(E_G)$. This
induced connection on $\text{ad}(E_G)$ will be denoted by 
$\nabla^{\rm ad}$. For
any point $x\, \in\, T$, consider the holomorphic vector
bundle
\begin{equation}\label{e3}
\text{ad}(E_G)^x\, :=\, \text{ad}(E_G)\vert_{f^{-1}(x)}
\, \longrightarrow\, f^{-1}(x)
\end{equation}
(see \eqref{e2} for $f$). Let $\nabla^x$ be the holomorphic 
connection on $\text{ad}(E_G)^x$ obtained by restricting the 
above connection $\nabla^{\rm ad}$.

Any complex Fano manifold is rationally connected
\cite[p. 766, Theorem 0.1]{KMM}. In particular,
$f^{-1}(x)$ is a rationally connected smooth complex
projective variety. Since
$M$ is rationally connected, the curvature of the connection
$\nabla^x$ vanishes identically (see \cite[p. 160,
Theorem 3.1]{Bi}). From the fact that $f^{-1}(x)$ is rationally
connected it also follows that $f^{-1}(x)$ is
simply connected \cite[p. 545, Theorem 3.5]{Ca},
\cite[p. 362, Proposition 2.3]{Ko}. Since $\nabla^x$ is flat, and 
$f^{-1}(x)$ is simply connected, we conclude that the vector
bundle $\text{ad}(E_G)^x$ in \eqref{e3} is holomorphically trivial.

Let
\begin{equation}\label{e4}
0\,\longrightarrow\, \text{ad}(E_H)\vert_{f^{-1}(x)}
\, \,\longrightarrow\,\text{ad}(E_G)^x
\,\stackrel{\alpha}{\longrightarrow}\,
(TM)\vert_{f^{-1}(x)} \,\longrightarrow\, 0
\end{equation}
be the restriction to $f^{-1}(x)\, \subset\, M$ of the short
exact sequence in \eqref{eq1}. Let $T_xT$ be the tangent space
to $T$ at the point $x$. The trivial vector bundle over
$f^{-1}(x)$ with fiber $T_xT$ will be denoted by
$f^{-1}(x)\times T_xT$. Let
$$
(df)\vert_{f^{-1}(x)}\, :\, (TM)\vert_{f^{-1}(x)}
\,\longrightarrow\, f^{-1}(x)\times T_xT
$$
be the differential of $f$ restricted to $f^{-1}(x)$.
The kernel of the composition homomorphism
$$
\text{ad}(E_G)^x\,\stackrel{\alpha}{\longrightarrow}\,
(TM)\vert_{f^{-1}(x)} \,
\stackrel{(df)\vert_{f^{-1}(x)}}{\longrightarrow}\, 
f^{-1}(x)\times T_xT
$$
(see \eqref{e4} for $\alpha$) will be denoted by ${\mathcal K}^x$.
So, from \eqref{e4} we get the short exact sequence of vector bundles
\begin{equation}\label{e5}
0\,\longrightarrow\, {\mathcal K}^x
\, \,\longrightarrow\,\text{ad}(E_G)^x
\,\longrightarrow\,f^{-1}(x)\times T_xT \,\longrightarrow\, 0
\end{equation}
over $f^{-1}(x)$.

Since both $\text{ad}(E_G)^x$ and $f^{-1}(x)\times T_xT$ are 
holomorphically trivial, using
\eqref{e5} it can be shown that the vector bundle
${\mathcal K}^x$ is also holomorphically trivial. To prove that
${\mathcal K}^x$ is also holomorphically trivial, fix a point
$z_0\, \in\, f^{-1}(x)$, and fix a subspace
\begin{equation}\label{e6}
V_{z_0}\, \subset\, \text{ad}(E_G)^x_{z_0}
\end{equation}
that projects isomorphically to the fiber of
$f^{-1}(x)\times T_xT$ over the point $z_0$. Since
$\text{ad}(E_G)^x$ is holomorphically trivial, there is a unique
holomorphically trivial subbundle
$$
V\, \subset\, \text{ad}(E_G)^x
$$ 
whose fiber over $z_0$ coincides with the subspace $V_{z_0}$ in
\eqref{e6}. Consider the homomorphism
$$
V\, \longrightarrow\, f^{-1}(x)\times T_xT
$$
obtained by restricting the projection in \eqref{e5}. Since this
homomorphism is an isomorphism over $z_0$, and both $V$ and
$f^{-1}(x)\times T_xT$ are holomorphically trivial, we conclude
that this homomorphism is an isomorphism over $f^{-1}(x)$.
Therefore, $V$ gives a holomorphic splitting of the short exact
sequence in \eqref{e5}.
Consequently, the vector bundle $\text{ad}(E_G)^x$ decomposes as
\begin{equation}\label{de}
\text{ad}(E_G)^x\,=\, {\mathcal K}^x\oplus V\, .
\end{equation}
Since $\text{ad}(E_G)^x$ is trivial, from a theorem of Atiyah on 
uniqueness of decomposition, \cite[p. 315, Theorem 2]{At1},
it follows that the vector bundle
${\mathcal K}^x$ is trivial; decompose all the three vector bundles
in \eqref{de} as direct sums of indecomposable vector bundles, and
apply Atiyah's result.
{}From \eqref{e4} we get a short exact sequence of holomorphic vector
bundles
\begin{equation}\label{e7}
0\,\longrightarrow\, \text{ad}(E_H)\vert_{f^{-1}(x)}
\, \,\longrightarrow\,{\mathcal K}^x
\,\longrightarrow\, T(f^{-1}(x)) \,\longrightarrow\, 0\, ,
\end{equation}
where $T(f^{-1}(x))$ is the holomorphic tangent bundle of
$f^{-1}(x)$. Since ${\mathcal K}^x$ is trivial, from \eqref{e7}
it follows that the tangent bundle $T(f^{-1}(x))$ is generated
by its global sections. This immediately implies that
$f^{-1}(x)$ is a homogeneous manifold.

Since $f^{-1}(x)$ is a Fano homogeneous manifold, we conclude
that there is a semisimple linear algebraic group
${\mathcal G}'$ over $\mathbb C$, and a parabolic subgroup
$P'\, \subset\, {\mathcal G}'$,
such that $f^{-1}(x)\,=\, {\mathcal G}'/P'$.
Since a quotient space of the type ${\mathcal G}'/P'$ is rigid
\cite[p. 131, Corollary]{Ak},
if follows that any two fibers of $f$ are
holomorphically isomorphic.

Let
\begin{equation}\label{e8}
{\mathcal G}\, :=\, \text{Aut}^0(f^{-1}(x))
\end{equation}
be the group of all holomorphic automorphisms of $f^{-1}(x)$.
It is known that ${\mathcal G}$ is a connected semisimple
complex linear algebraic group \cite[p. 131, Theorem 2]{Ak}.
Since $f^{-1}(x)$ is isomorphic to ${\mathcal G}'/P'$, it
follows that ${\mathcal G}$ is a semisimple linear algebraic group
over $\mathbb C$ of adjoint type (this means that the center of
$\mathcal G$ is trivial). As before, let
\begin{equation}\label{e8a}
z_0\, \in\, f^{-1}(x)
\end{equation}
be a fixed point. Let
\begin{equation}\label{e9}
{\mathcal P}\, \subset\, {\mathcal G}
\end{equation}
be the subgroup that fixes the point $z_0$. Note that $\mathcal P$
is a parabolic subgroup of $\mathcal G$, and the quotient
${\mathcal G}/P$ is identified with $f^{-1}(x)$.

Consider the trivial holomorphic fiber bundle
$$
T\times f^{-1}(x)\, \longrightarrow\, T
$$
with fiber $f^{-1}(x)$.
Let ${\mathcal E}\, \longrightarrow\, T$ be the holomorphic
fiber bundle given by the sheaf of holomorphic isomorphisms from
$T\times f^{-1}(x)$ to $M$, where
$M$ is the fiber bundle in \eqref{e1}; recall that
all the fibers of $f$ are holomorphically isomorphic. It is 
straightforward
to check that $\mathcal E$ is a holomorphic principal
${\mathcal G}$--bundle, where $\mathcal G$ is the group defined
in \eqref{e8}. Let
\begin{equation}\label{e10}
\varphi\,:\, {\mathcal E}_{\mathcal G}\, :=\,
{\mathcal E}\,\longrightarrow\, T
\end{equation}
be this holomorphic principal ${\mathcal G}$--bundle.
The fiber of ${\mathcal E}_{\mathcal G}$ over any point
$y\,\in\, T$ is the space of all holomorphic
isomorphisms from $f^{-1}(x)$ to $f^{-1}(y)$.

So there is a natural projection
\begin{equation}\label{np}
{\mathcal E}_{\mathcal G}\, \longrightarrow\, M
\end{equation}
that sends any $\xi\, \in\, \varphi^{-1}(y)$ to the
image of the point $z_0$ in \eqref{e8a} by the
map
$$
\xi\,:\, f^{-1}(x) \,\longrightarrow\, 
f^{-1}(y)\, .
$$
This projection identifies the fiber bundle
$$
{\mathcal E}_{\mathcal G}/{\mathcal P}\,\longrightarrow\, T
$$
with the fiber bundle $M\, \longrightarrow\, T$,
where $\mathcal P$ is the subgroup in \eqref{e9}. This
completes the proof of the theorem.
\end{proof}

\section{Principal bundles over a torus}

Let $G_0$ be a reductive linear algebraic group
defined over $\mathbb C$.
Fix a maximal compact subgroup
$
K_0\, \subset\, G\, .
$
Let $Y$ be a complex manifold and $E_{G_0}\, \longrightarrow\, Y$
a holomorphic principal $G_0$--bundle over $Y$. A \textit{unitary
flat connection} on $E_{G_0}$ is a flat holomorphic connection
$\nabla^0$ on $E_{G_0}$ which has the following property: there
is a $C^\infty$ reduction of structure group
$
E_{K_0}\,\subset\, E_{G_0}
$
of $E_{G_0}$ to the subgroup $K_0$ such that $\nabla^0$ is induced
by a connection on $E_{K_0}$ (equivalently, the connection $\nabla^0$
preserves $E_{K_0}$). Note that $E_{G_0}$ admits a unitary flat
connection if and only if $E_{G_0}$ is given by a homomorphism
$\pi_1(Y)\, \longrightarrow\, K_0$.

Let $P\, \subset\, G_0$ be a parabolic subgroup. Let
$R_u(P)\, \subset\, P$ be the unipotent radical. The
quotient group $L(P)\, :=\, P/R_u(P)$, which is
called the Levi quotient of $P$, is reductive (see \cite[p. 158,
\S~11.22]{Bo}). Given a holomorphic principal $P$--bundle
$E_{P}\, \longrightarrow\, Y$, let
$$
E_{L(P)}\, :=\, E_{P}\times^P L(P)\, \longrightarrow\, Y
$$
be the principal $L(P)$--bundle obtained by extending the
structure group of $E_{P}$ using the quotient map
$P\, \longrightarrow\, L(P)$. Note that $E_{L(P)}$ is
identified with the quotient $E_{P}/R_u(P)$. By a
\textit{unitary flat connection} on $E_{P}$ we will mean
a unitary flat connection on the principal $L(P)$--bundle
$E_{L(P)}$ (recall that $L(P)$ is reductive).

A vector bundle $E\,\longrightarrow \, Y$ is called
\textit{numerically flat} if both $E$ and its dual $E^*$
are numerically effective \cite[p. 311, Definition 1.17]{DPS}.

\begin{proposition}\label{prop2}
Let $E_{G_0}$ be a holomorphic principal $G_0$--bundle over a
compact connected K\"ahler manifold $Y$. Then the following
four statements are equivalent:
\begin{enumerate}
\item There is a parabolic proper subgroup $P\, \subset\, G_0$
and a strictly anti--dominant character $\chi$ of $P$ such that
the associated line bundle
$$
E_{G_0}(\chi)\, :=\, E_{G_0}\times^{P} {\mathbb C}
\,\longrightarrow \, E_{G_0}/P
$$
is numerically effective.

\item The adjoint vector bundle ${\rm ad}(E_{G_0})$ is numerically
flat.

\item The principal $\mathcal G$--bundle $E_{G_0}$
is pseudostable, and $c_2({\rm ad}(E_{G_0}))\,=\, 0$
(see \cite[p. 26, Definition 2.3]{BG} for the definition of 
pseudostability).

\item There is a parabolic subgroup $P_0\, \subset\,
G_0$ and a holomorphic reduction of structure group
$
E_{P_0}\, \subset\, E_{G_0}
$
of $E_{G_0}$ such that $E_{P_0}$ admits a unitary flat connection.
\end{enumerate}
\end{proposition}

\begin{proof}
This proposition follows from \cite[p. 154, Theorem 1.1]{BS}
and \cite[Theorem 1.2]{BB}.
\end{proof}

\begin{lemma}\label{lem2}
Let $T$ be a complex torus, and let $E_{G_0}\,\longrightarrow\, T_0$
be a holomorphic principal $G_0$--bundle.
Let $P\, \subset\, G_0$ be a parabolic subgroup. If the four equivalent
statements in Proposition \ref{prop2} hold, then the holomorphic tangent
bundle of $E_{G_0}/P$ is numerically effective.
\end{lemma}

\begin{proof}
Assume that the four equivalent
statements in Proposition \ref{prop2} hold.

Let $\delta\, :\, E_{G_0}/P\, \longrightarrow\, T_0$ be the natural 
projection.
Let
$$
T_\delta\, :=\, {\rm kernel}(d\delta)\, \subset\, T(E_{G_0}/P)
$$
be the relative tangent bundle for the projection $\delta$. The vector
bundle $T_\delta\,\longrightarrow\,
E_{G_0}/P$ is a quotient of the adjoint vector bundle
${\rm ad}(E_{G_0})$. Since ${\rm ad}(E_{G_0})$ is numerically
effective (second statement in Proposition \ref{prop2}),
it follows that $T_\delta$ is numerically effective
\cite[p. 308, Proposition 1.15(i)]{DPS}.

Consider the short exact sequence of vector bundles on $E_{G_0}/P$
$$
0\,\longrightarrow\, T_\delta \,\longrightarrow\, T(E_{G_0}/P)
\,\stackrel{d\delta}{\longrightarrow}\, \delta^*TT_0
\,\longrightarrow\, 0\, .
$$
Since $\delta^*TT_0$ and $T_\delta$ are numerically effective
($TT_0$ is trivial), it follows that $T(E_{G_0}/P)$ is numerically
effective \cite[p. 308, Proposition 1.15(ii)]{DPS}.
This completes the proof of the lemma.
\end{proof}

As before, $X$ is a compact connected K\"ahler manifold such 
that $TX$ is numerically effective, and $(E'_H\, ,\theta')$
be a holomorphic Cartan geometry on $X$ of type $G/H$. Also,
$\gamma$ and $f$ are the maps constructed in \eqref{e1}
and \eqref{e2} respectively. Let
\begin{equation}\label{g1}
K^{-1}_f\, \longrightarrow\, M
\end{equation}
be the relative anti--canonical line bundle for the projection $f$.

Let $\mathcal G$ be the group in \eqref{e8},
and let ${\mathcal E}_{\mathcal G}\, \longrightarrow\,
T$ be the principal $\mathcal G$--bundle constructed in
\eqref{e10}. Let
$
\text{ad}({\mathcal E}_{\mathcal G})\,\longrightarrow\, T
$
be the adjoint vector bundle.

\begin{proposition}\label{lem1}
Let $X$ is a compact connected K\"ahler manifold such
that $TX$ is numerically effective, and let
$(E'_H\, ,\theta')$
be a holomorphic Cartan geometry on $X$ of type $G/H$. Then
the relative anti--canonical line bundle $K^{-1}_f$ in \eqref{g1}
is numerically effective. Also, the following three statements hold:
\begin{enumerate}
\item The adjoint vector bundle ${\rm ad}({\mathcal E}_{\mathcal G})$
is numerically flat.

\item The principal $\mathcal G$--bundle
${\mathcal E}_{\mathcal G}$ is pseudostable, and
$c_2({\rm ad}({\mathcal E}_{\mathcal G}))\,=\, 0$.

\item There is a parabolic subgroup ${\mathcal P}\, \subset\,
{\mathcal G}$ and a holomorphic reduction of structure group
$
{\mathcal E}_{\mathcal P}\, \subset\, {\mathcal E}_{\mathcal G}
$
of ${\mathcal E}_{\mathcal G}$ such that ${\mathcal E}_{\mathcal P}$
admits a unitary flat connection.
\end{enumerate}
\end{proposition}

\begin{proof}
Let $\gamma\,:\, M\,\longrightarrow\, X$ be
the covering in \eqref{e1}, and let
$
f\, :\, M\, \longrightarrow\, T
$
be the projection in \eqref{e2}. There is a 
semisimple complex linear algebraic group $\mathcal
G$, a parabolic subgroup $P\, \subset\, \mathcal G$, and
a holomorphic principal $\mathcal G$--bundle
${\mathcal E}_{\mathcal G}\, \longrightarrow\, T$
such that the fiber bundle ${\mathcal E}_{\mathcal G}/P\,
\longrightarrow\, T$ is holomorphically isomorphic to the
one given by $f$ (see Theorem \ref{thm1}).

Since the canonical line bundle $K_T\, \longrightarrow\, T$ is
trivial, the line bundle $K^{-1}_f$ is isomorphic to
$K^{-1}_M$. The anti--canonical line bundle $K^{-1}_M$ is numerically
effective because $TM$ is numerically effective. Hence
$K^{-1}_f$ is numerically effective. Recall that
${\mathcal E}_{\mathcal G}/{\mathcal P}\,=\, M$ using the projection
in \eqref{np}. The line bundle $K^{-1}_f$ corresponds to a
strictly anti--dominant character of $\mathcal P$ because
$K^{-1}_f$ is relatively ample. Hence the first of
the four statements in Proposition \ref{prop2} holds. Now
Proposition \ref{prop2} completes the proof of the proposition.
\end{proof}

\begin{proposition}\label{lem3}
Let $X$ and $(E'_H\, ,\theta')$ be as in Lemma \ref{lem1}.
The principal $\mathcal G$--bundle ${\mathcal E}_{\mathcal G}$ 
constructed in Theorem \ref{thm1} admits a flat holomorphic
connection.
\end{proposition}

\begin{proof}
We know that principal $\mathcal G$--bundle
${\mathcal E}_{\mathcal G}$ is pseudostable, and
$c_2({\rm ad}({\mathcal E}_{\mathcal G}))\,=\, 0$ (see the
second statement in Proposition \ref{lem1}). Hence the
proposition follows from \cite[p. 20, Theorem 1.1]{BG}.
\end{proof}



\begin{thebibliography}{1111}\frenchspacing\small

\bibitem{Ak} D. N. Akhiezer, \textit{Lie Groups Actions in
Complex Analysis}, Aspects of Mathematics, E27, Friedr. Vieweg
\& Sohn, Braunschweig, 1995.

\bibitem{At1} M. F. Atiyah, On the Krull--Schmidt theorem
with application to sheaves, Bull. Soc. Math. Fr. \textbf{84}
(1956), 307--317.

\bibitem{At2} \leavevmode\vrule height 2pt depth -1.6pt width 23pt,
 Complex analytic connections in fibre
bundles, Trans. Amer. Math. Soc. \textbf{85} (1957) 181--207.

\bibitem{Bi} I. Biswas, Principal bundle, parabolic bundle,
and holomorphic connection, in: \textit{A tribute to C. S. Seshadri
(Chennai, 2002)}, 154--179, Trends Math., Birkh\"auser, Basel,
2003.

\bibitem{BB} I. Biswas and U. Bruzzo, On semistable principal bundles
over a complex projective manifold, II, Geom. Dedicata {\bf  146} (2010), 27--41.

\bibitem{BG} I. Biswas and T. L. G\'omez, Connections and Higgs
fields on a principal bundle, Ann. Glob.
Anal. Geom. \textbf{33} (2008), 19--46.

\bibitem{BS} I. Biswas and G. Schumacher, Numerically effectiveness
and principal bundles on K\"ahler manifolds, Ann. Glob. Anal. Geom.
\textbf{34} (2008), 153--165.

\bibitem{Bo} A. Borel, \textit{Linear Algebraic Groups},
Graduate Texts in Mathematics, No. 126, Springer--Verlag,
New--York, 1991.

\bibitem{Ca} F. Campana, On twistor spaces of the class
$\mathcal C$, Jour. Diff. Geom. \textbf{33}
(1991) 541--549.

\bibitem{CP} F. Campana and T. Peternell, Projective manifolds
whose tangent bundles are numerically effective, Math. Ann.
\textbf{289} (1991), 169--187.

\bibitem{DPS} J.-P. Demailly, T. Peternell and M. Schneider, Compact
complex manifolds with numerically effective tangent bundles, Jour.
Alg. Geom. \textbf{3} (1994), 295--345.

\bibitem{Ko} J. Koll\'ar, Fundamental groups of rationally connected
varieties, Michigan Math. Jour. \textbf{48} (2000) 359--368.

\bibitem{KMM} J. Koll\'ar, Y. Miyaoka and S. Mori, Rational 
connectedness and boundedness for Fano manifolds, Jour. Diff.
Geom. \textbf{36} (1992), 765--779.

\bibitem{S} R. W. Sharpe, \textit{Differential Geometry},
Springer-Verlag, Heidelberg, 1997.

\end{thebibliography}
\end{document}